\documentclass[preprint,12pt]{elsarticle}

 \usepackage{graphicx}
\usepackage[margin=1.0in]{geometry} 

\usepackage{amsmath} 
\usepackage{amssymb}
\usepackage{amsthm}

\newtheorem{thm}{Theorem}
\newtheorem{lem}[thm]{Lemma}
\newtheorem{proposition}[thm]{Proposition}

\newdefinition{rmk}{Remark}

\usepackage{hyperref}
\usepackage[usenames, dvipsnames]{color}

\usepackage{fancyhdr}
\usepackage{lipsum} 
\usepackage[linesnumbered,ruled,vlined]{algorithm2e}

\journal{arXiv}

\begin{document}

\begin{frontmatter}



\title{Efficient Counting of Degree Sequences}


\author{Kai Wang\corref{corkw}}
\ead{kwang@georgiasouthern.edu}
\address{Department of Computer Sciences, Georgia Southern University, Statesboro, GA, USA, 30458}

\cortext[corkw]{Corresponding author}


\begin{abstract}
Novel dynamic programming algorithms to count the set $D(n)$ of zero-free degree sequences of length $n$,
the set $D_c(n)$ of degree sequences of connected graphs on $n$ vertices and the set $D_b(n)$ of degree sequences
of biconnected graphs on $n$ vertices exactly are presented.
They are all based on a recurrence of Barnes and Savage and shown to run in polynomial time
and are asymptotically much faster than the previous best known algorithms for these problems.
These appear to be the first polynomial time algorithms to compute $|D(n)|$, $|D_c(n)|$ and $|D_b(n)|$
to the author's knowledge and have enabled us to
tabulate them up to $n=118$, the majority of which were unknown. The available numerical results of $|D(n)|$
tend to give more supporting evidence of a conjecture of Gordon F. Royle about the limit
of $|D(n)|/|D(n-1)|$.
The OEIS entries that can be computed by algorithms in this paper are A004251, A007721, A007722 and A095268.
\end{abstract}


\end{frontmatter}



\section{Introduction}
\label{sec:Intro}
The method to count the number of unlabeled graphs, unlabeled connected graphs and unlabeled biconnected graphs
on $n$ vertices exactly based on the Redfield-P\'{o}lya
enumeration theorem is well-known and proves to be efficient \cite{Redfield1927,Polya1937,Robinson1970,HararyPalmer1973}.
However, there was no known efficient method to count the number
of degree sequences of these graphs exactly. We present a solution to these problems in this paper. Only simple
(i.e. with no loops or multiple edges) undirected graphs and their degree sequences are considered.

A degree sequence of length $n$ is a sequence of nonnegative integers $d_1, d_2, \ldots, d_n$ (in which order is
irrelevant so that they are usually written in non-increasing order without loss of generality)
such that there is a simple undirected graph $G$ of $n$ vertices
with the vertex degrees $d_1, d_2, \ldots, d_n$.
A zero-free degree sequence is a degree sequence such that each term of the sequence is positive, or equivalently,
a degree sequence of a simple undirected graph with no isolated vertices. For convenience, from now on
by degree sequences we mean \textit{zero-free} degree sequences unless otherwise noted. Let $D(n)$ denote the set
of degree sequences of length $n$. A positive integer sequence $a_1, a_2, \ldots, a_n$ (in which order is also irrelevant)
with sum $N=\sum_{i=1}^{n}a_i$ is called a graphical partition of $N$ with $n$
parts, or a graphic sequence, if it is the degree sequence of a simple undirected graph.
Note that $N$ must be even to have a graphical partition, since the sum of the vertex degrees of a simple
undirected graph is twice the number of its edges and therefore must be even. Sierksma and Hoogeveen \cite{Sierksma1991}
showed seven criteria to be equivalent to characterize graphic integer sequences, of which the
Erd{\H{o}}s-Gallai criterion \cite{ErdosCallai1960} is probably the most widely known.

The problems of counting or enumerating the set $G(N)$ of graphical partitions of a fixed positive even integer $N$ and the set
$D(n)$ for a fixed length $n$ have been extensively studied, for example \cite{Erdos1993,Ruskey1994,BarnesSavage1995,
BarnesSavage1997,Pittel1999,Kohnert2004,Burns2007,Ivanyi2013,Pittel2016}. Pittel \cite{Pittel1999} affirmatively resolved
a problem originally posed in 1982 by Herbert S. Wilf whether the fraction of
graphical partitions in the set $P(N)$ of all partitions of a positive integer $N$ goes to 0 as $N$
approaches infinity. Erd{\H{o}}s and Richmond \cite{Erdos1993} showed that this fraction as a function of $N$
cannot tend to 0 faster than $\pi/\sqrt{6N}$. Recently Pittel \cite{Pittel2016} has obtained an upper bound
of $\exp({-0.11\log N/\log \log N})$ for $|G(N)|/|P(N)|$
with $N$ sufficiently large using a local limit theorem. Comparison between the bound and actual values
of $|G(N)|/|P(N)|$ up to $N=910$ suggests that the bound does not appear to be tight. Up till now, the
asymptotic order of $|G(N)|$ is still unknown, though the order of $|P(N)|$ has been known to be $|P(N)|\sim \frac{\exp(\pi\sqrt{2N/3})}{4\sqrt{3}N}$
since 1918 by Hardy and Ramanujan \cite{Hardy1918}.

For the asymptotic growth rate of $|D(n)|$, good upper and lower bounds are known. Burns \cite{Burns2007} showed there
are positive constants $c_1$ and $c_2$ such that $4^n/(c_1 n) \le |D(n)| \le 4^n/((\log n)^{c_2} \sqrt{n})$ for all sufficiently
large $n$. Ruskey \textit{et al.} \cite{Ruskey1994} presented a fast algorithm to enumerate degree sequences of length $n$ using
the ``Reverse Search'' approach originated by Avis and Fukuda \cite{Avis1996}. The algorithm appears to run in constant
amortized time (CAT), though no proof has been given yet to the knowledge of the author. No matter whether
it is truly CAT or not, it is clear that its time complexity is at least $\Omega(4^n/n)$ based on the lower bound of $|D(n)|$ as proved
in \cite{Burns2007}. Exact values of $|D(n)|$ up to $n=31$ were obtained in
Iv\'{a}nyi \textit{et al.} \cite{Ivanyi2013} by testing every sequence
in the set $E(n)$ of even bounded sequences of length $n$ (that is, a non-increasing positive integer sequence of length $n$
with even sum and largest term less than $n$) for its graphicness. A sequence in $E(n)$
can be tested whether it is graphic in linear time and $|E(n)|$ is asymptotically
$4^n/(8\sqrt{\pi n})$ \cite{Ivanyi2013}. Therefore, this filtering method of calculating $|D(n)|$ has an asymptotic time complexity of $\Theta(\sqrt{n} 4^n)$
and space complexity of $O(n)$. By our own testing, the method in \cite{Ruskey1994} appears to
be faster than the filtering method in \cite{Ivanyi2013} if both are restricted to run in single-threaded mode,
which would make sense because the``Reverse Search'' based approach would have a time complexity of $O(4^n/((\log n)^{c_2} \sqrt{n}))$
under the assumption that it runs in CAT.

The OEIS entry \textcolor{red}{\href{https://oeis.org/A095268}{A095268}}
records the values of $|D(n)|$. 
In 2006 Gordon F. Royle posed the question whether $|D(n)|/|D(n-1)|$ tend to the limit 4 when $n$ approaches infinity
\cite{OEISA095268} as there seems to be no apparent reason why $|D(n)|/|D(n-1)|$ would tend to the same limit as $|E(n)|/|E(n-1)|$.
In this paper we will present two fast novel algorithms to calculate $|D(n)|$ based on several recurrences,
circumventing the need to generate all degree sequences with a given length as performed in \cite{Ruskey1994,Ivanyi2013}.
We will show that both of these algorithms run in polynomial time, which appear to be the first such algorithms for this problem and
are asymptotically faster than previous exponential algorithms shown above, at the expense of larger space requirement.
The numerical results of $|D(n)|$ so obtained up to $n=118$ tend to agree with the conjecture posed by Royle.
Similar techniques are utilized to design fast polynomial time algorithms to count the number of degree sequences of connected
and biconnected graphs on $n$ vertices. The only known previous algorithms to count them seem to be first generating
all the degree sequences in $D(n)$ and then counting a subset satisfying the relevant connected and biconnected conditions, which make them all
run in exponential time. A related problem of counting the number of degree sequences with fixed degree sum (instead of fixed length) of connected and biconnected
\textit{multigraphs} that might contain multiple edges is considered by R{\o}dseth \textit{et al.} \cite{Rodseth2009}.

We note that the degree sequences considered in this paper are unordered. For example, the degree sequences $211$, $121$
and $112$ of the path graph with three vertices are all treated as the same. When regarded as ordered degree sequences, they
are considered distinct. Ordered partitions are called compositions in the literature \cite{Andrews1984}. The efficient
computation of the number $f(n)$ of ordered degree sequences of length $n$ allowing zero terms
(the OEIS entry \textcolor{red}{\href{https://oeis.org/A005155}{A005155}}) has been solved by
Stanley \cite{Stanley1991} using the theory of zonotopes. For example, the exact value of $f(200)$ is about
$9.97\times 10^{458}$. The asymptotic order of $f(n)$ is shown on its OEIS entry by Vaclav Kotesovec 
to be $\frac{\Gamma(3/4)}{2^{3/4}\sqrt{e \pi}}n^{n-1/4}(1-\frac{11\pi}{24\sqrt{n}(\Gamma(3/4))^2})$, which is also asymptotic
to $\frac{\Gamma(3/4)}{2^{3/4}\sqrt{e \pi}}n^{n-1/4}$. By contrast, the tight asymptotic order of $|D(n)|$ is unknown
and we are unable to resolve it at this time.

The rest of the paper is organized as follows. In Section \ref{sec:Algo}, we formalize the necessary definitions,
derive the relevant formulae and describe the two algorithms to compute $|D(n)|$. 
In Section \ref{sec:count_connected} we present algorithms based on similar techniques to compute the number of degree sequences
of connected graphs on $n$ vertices. In Section \ref{sec:count_biconnected} we present an algorithm based on similar dynamic
programming techniques and additional combinatorial observations to compute the number of degree sequences
of biconnected graphs on $n$ vertices. In Section \ref{sec:Analysis} we analyze the computational complexity of the algorithms.
In Section \ref{sec:experiments} we give some brief computational results that serve as proof of concept.
In Section \ref{sec:conclusion} we conclude and suggest further research directions.

\section{Description of the Two Algorithms for $|D(n)|$}
\label{sec:Algo}
In this section we first review the definitions needed to describe the two algorithms to calculate $|D(n)|$.
We then present the relevant formulae for $|D(n)|$ and the pseudo-codes of our algorithms.
For the reader's convenience, the terminology employed in this paper is summarized in Table \ref{tbl:definitions}.
\begin{table}[!htb]
	\centering
	\caption{Terminology used in this paper}
	\begin{tabular}{||c|l||}
		\hline\hline
		Term & Meaning\\
		\hline\hline
		$P(N)$ & set of all partitions of an integer $N$\\
		\hline
		$P(N,k,l)$ & set of partitions of an integer $N$ into at most $l$ parts\\
				   & with largest part at most $k$ \\
		\hline
		$P(N,k,l,s)$ & subset of $P(N,k,l)$ determined by integer $s$ \\
		\hline
		$P^{'}(N,k,l)$ & set of partitions in $P(N,k,l)$ with exactly $l$ parts\\
						& and largest part exactly $k$\\
		\hline
		$P^{'}(N,k,l,s)$ & set of partitions in $P(N,k,l,s)$ with exactly $l$ parts\\
						& and largest part exactly $k$\\
		\hline
		$G(N)$ & set of all graphical partitions of an even integer $N$\\
		\hline
		$G^{'}(N,k,l)$ & set of graphical partitions in $P^{'}(N,k,l)$ \\
		\hline
		$G^{'}(N,l)$ & set of graphical partitions of $N$ with exactly $l$ parts\\
		\hline
		$H^{'}(N,l)$ & set of graphical partitions of $N$ with exactly $l$ parts\\
		& and largest part exactly $l-1$\\
		\hline
		$L^{'}(N,l)$ & set of graphical partitions of $N$ with exactly $l$ parts\\
		& and largest part less than $l-1$\\
		\hline\hline
		$d(\pi)$ & size of the Durfee square of the partition $\pi$ \\
		\hline
		$r(\pi)$ & corank vector of the partition $\pi$ \\
		\hline\hline
		$D(n)$ & set of (zero-free) degree sequences of length $n$ \\
		\hline
		$B(n)$ & subset of $D(n)$ with largest part exactly $n-1$ and\\
		& smallest part exactly $1$\\
		\hline
		$C(n)$ & subset of $D(n)$ with smallest part exactly $1$ \\
		\hline
		$D_0(n)$ & set of degree sequences of length $n$ allowing zero terms \\
		\hline
		$D_b(n)$ & set of degree sequences of biconnected graphs on $n$ vertices \\
		\hline
		$D_c(n)$ & set of degree sequences of connected graphs on $n$ vertices \\
		\hline
		$D_d(n)$ & subset of $D(n)$ that are forcibly disconnected \\
		\hline
		$E(n)$ & set of even bounded sequences of length $n$ \\
		\hline
		$f(n)$ & number of ordered degree sequences of length $n$ \\
		& allowing zero terms\\
		\hline
		$H(n)$ & subset of $D(n)$ with largest part exactly $n-1$ \\
		\hline
		$L(n)$ & subset of $D(n)$ with largest part less than $n-1$ \\
		\hline
		$S(n)$ & subset of $D(n)$ with largest part exactly $n-2$ \\
		\hline
		$W(n)$ & subset of $D(n)$ with largest part less than $n-1$ and\\
		& smallest part exactly 1 \\
		\hline
		$I_e(n)$&$\{N|n\le N\le n(n-1), N \mbox{ is an even integer}\}$\\
		\hline
		$I'_e(n)$&$\{N|2(n-1)\le N\le n(n-1), N \mbox{ is an even integer}\}$\\
		\hline
		$I''_e(n)$&	$\{N|2n-3\le N\le n(n-2), N \mbox{ is an even integer}\}$\\
		\hline
		$J_e(n)$&$\{N|n\le N\le n(n-2), N \mbox{ is an even integer}\}$\\
		\hline
		$J'_e(n)$&$\{N|n\le N< n(n-1)/2, N \mbox{ is an even integer}\}$\\
		\hline\hline
	\end{tabular}
	\label{tbl:definitions}
\end{table}

\subsection{Definitions and the Basic Algorithm}
\label{subsec:BasicAlgo}
The notion of \textit{partitions} in number theory is standard and can be found in many textbooks, for example \cite{Hardy2008,Andrews1984}.
A partition of an integer is a way of writing the integer as a sum of positive integers in which the order
of the summands (also called parts) is irrelevant. The parts in a partition are usually written
in non-increasing order, for example, $8=4+3+1$.
A partition can be visually represented by its Ferrers diagram as a left-aligned array of dots so that the
number of dots in the $i$-th row of the array is equal to the $i$-th part of the partition. The Durfee square
is the upper-left largest square that is embedded within a partition's Ferrers diagram. Figure \ref{fig:part8}
shows the Ferrers diagram of the example partition $8=4+3+1$. This partition has three parts, the largest part
is 4, and its Durfee square contains two rows and two columns.

\begin{figure}[!htbp]
	\centering
	\includegraphics[width=4cm]{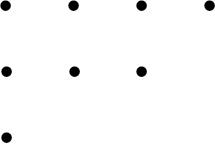}\\
	\caption{Ferrers diagram of the partition 8=4+3+1. Its Durfee square size is 2.}
	\label{fig:part8}
\end{figure}

Following the notation of Barnes and Savage \cite{BarnesSavage1995}, let $P(N,k,l)$ denote the set of partitions
of an integer $N$ into at most $l$ parts with largest part at most $k$, and for integer $s$, define $P(N,k,l,s)$ by
\[ P(N,k,l,s)=\left\{ \begin{array}{ll}
\emptyset & \mbox{if $s < 0$};\\
\{\pi \in P(N,k,l) | s+\sum_{i=1}^{j}r_i(\pi)\ge j, 1 \le j \le d(\pi)\} & \mbox{if $s \ge 0$},\end{array} \right.\]
where $d(\pi)$ is the size (number of rows) of the Durfee square of the partition $\pi$ and 
$r(\pi)=(r_1(\pi),\ldots,r_{d(\pi)}(\pi))$ is the corank vector of $\pi$ with $r_i(\pi)=\pi_i^{'}-\pi_i$ and
$\pi_i$ and $\pi_i^{'}$ are the number of dots in the $i$-th row and column of the Ferrers diagram of $\pi$,
respectively, for $1 \le i \le d(\pi)$.
The example partition $8=4+3+1$ in Figure \ref{fig:part8} has the corank vector $(-1,-1)$.

Additionally, as in \cite{BarnesSavage1995}, define $P^{'}(N,k,l)$ and $P^{'}(N,k,l,s)$
to be the subsets of $P(N,k,l)$ and $P(N,k,l,s)$, respectively, consisting of those partitions
with \textit{exactly} $l$ parts and largest part \textit{exactly} $k$.
Define $G^{'}(N,k,l)$ to be the set of graphical partitions in $P^{'}(N,k,l)$. Then we have
the following lemma that establishes a relationship between a set of graphical partitions
and a set of partitions.
\begin{lem}
\label{lem_graphical_exact}
For even integer $N \ge 0$, $G^{'}(N,k,l)=P^{'}(N,k,l,0)$.
\end{lem}
\begin{proof}
From the definition of $P^{'}(N,k,l,s)$, we have
\[P^{'}(N,k,l,0)=\{\pi \in P^{'}(N,k,l) |\sum_{i=1}^{j}r_i(\pi)\ge j, 1 \le j \le d(\pi)\}.\] 
From the definition of $G^{'}(N,k,l)$, we have
\[ G^{'}(N,k,l)=\{\pi \in P^{'}(N,k,l) |\pi \mbox{ is graphical}\}.\]
Now based on the Nash-Williams condition \cite{Ruch1979,Rousseau1995,Sierksma1991}, a partition $\pi$
of an even integer is graphical
if and only if $\sum_{i=1}^{j}r_i(\pi)\ge j, 1 \le j \le d(\pi)$. So we have $G^{'}(N,k,l)=P^{'}(N,k,l,0)$.

\end{proof}

Before presenting the algorithms for $|D(n)|$ we need the following lemma that connects the function $|P^{'}(*,*,*,*)|$
to its relative $|P(*,*,*,*)|$ and it will be
used to bridge the gap between the problem of counting graphical partitions and that of counting partitions.
Its proof is based on the observation that there is a bijection between $P^{'}(N,k,l)$ and $P(N-k-l+1,k-1,l-1)$.
\begin{lem}
\label{lem_partition_exact}
\cite[Lemma 4]{BarnesSavage1995} For $N>0,1\le k,l \le N, s\ge 0$,
\[|P^{'}(N,k,l,s)|=|P(N-k-l+1,k-1,l-1,s+l-k-1)|.\]
\end{lem}

The following theorem materializes a connection between particular variants of partition counting functions and
graphical partition counting functions. It is important to our algorithms to calculate $|D(n)|$.
\begin{thm}
	\label{thm_graphical_partition}
		For even integer $N \ge 0$, $|G^{'}(N,k,l)|=|P(N-k-l+1,k-1,l-1,l-k-1)|$.
\end{thm}
\begin{proof}
	Immediate from Lemma \ref{lem_graphical_exact} and Lemma \ref{lem_partition_exact}.
\end{proof}

There are tons of efficient algorithms to calculate various partition functions in the literature.
Our algorithms for $|D(n)|$ need an elegant recurrence of Barnes and Savage \cite{BarnesSavage1995} that makes it easy
to design an efficient dynamic programming algorithm to calculate all $|P(*,*,*,*)|$ values.
We now derive the simple connection between $D(n)$ and	$G^{'}(N,k,l)$ and then present our basic algorithm
to calculate $|D(n)|$ that eventually uses those $|P(*,*,*,*)|$ values.

Recall that $D(n)$ denotes the set of degree sequences of length $n$. Since each term in a degree sequence
of length $n$ is between 1 and $n-1$ inclusive, it follows that each degree sequence in $D(n)$ can be
seen as a graphical partition of some even integer $N$ with exactly $n$ parts and largest part exactly $k$
where $n\le N \le n(n-1)$ and $1\le k \le n-1$. Define $G^{'}(N,l)$ to be the set of graphical partitions of $N$
with \textit{exactly} $l$ parts. Also define an index set $I_e(n)=\{N|n\le N\le n(n-1), N \mbox{ is an even integer}\}$.
Based on the above observations we clearly have
\begin{equation} \label{eqn:set_deg_seq_edge}
D(n)=\bigcup_{N\in I_e(n)}G^{'}(N,n)
\end{equation}
and
\begin{equation} \label{eqn:set_deg_seq_edge_largest_part}
G^{'}(N,n)=\bigcup_{k=1}^{n-1}G^{'}(N,k,n),
\end{equation}
which can be combined into
\begin{equation} \label{eqn:set_deg_seq_edge_largest_part2}
D(n)=\bigcup_{N\in I_e(n)}\bigcup_{k=1}^{n-1}G^{'}(N,k,n).
\end{equation}
These are all disjoint unions so that we have
\begin{equation} \label{eqn:deg_seq_edge}
	\begin{aligned} 
|D(n)|=\sum_{N\in I_e(n)}|G^{'}(N,n)|,
	\end{aligned}
\end{equation}
\begin{equation} \label{eqn:graphical_partition}
	\begin{aligned} 
|G^{'}(N,n)|=\sum_{k=1}^{n-1}|G^{'}(N,k,n)|,
	\end{aligned}
\end{equation}
\begin{equation} \label{eqn:deg_seq_combine1}
\begin{aligned} 
|D(n)|=\sum_{N\in I_e(n)}\sum_{k=1}^{n-1}|G^{'}(N,k,n)|.
\end{aligned}
\end{equation}
Substituting the formula in Theorem \ref{thm_graphical_partition} into Equation (\ref{eqn:graphical_partition})
and Equation (\ref{eqn:deg_seq_combine1}) we get
\begin{equation} \label{eqn:graphical_partition_detailed}
\begin{aligned} 
|G^{'}(N,n)|=\sum_{k=1}^{n-1}|P(N-k-n+1,k-1,n-1,n-k-1)|,
\end{aligned}
\end{equation}
\begin{equation} \label{eqn:deg_seq_combine}
	\begin{aligned} 
|D(n)|=\sum_{N\in I_e(n)}\sum_{k=1}^{n-1}|P(N-k-n+1,k-1,n-1,n-k-1)|.
	\end{aligned}
\end{equation}
With Equation (\ref{eqn:deg_seq_combine}) it is sufficient to calculate
a finite number of $|P(*,*,*,*)|$ values to compute $|D(n)|$. As mentioned above, all these $|P(*,*,*,*)|$ values
can be calculated with a recurrence of Barnes and Savage in the following theorem:

\begin{thm}
\label{thm_partition}
\cite[Theorem 1]{BarnesSavage1995} $|P(N,k,l,s)|$ is defined by:
\begin{equation*}
\begin{aligned}
&|P(N,k,l,s)|:= \\
& \quad if((N<0) \mbox{ }or\mbox{ } (k<0) \mbox{ }or\mbox{ } (l<0) \mbox{ }or\mbox{ } (s<0)) \mbox{ } \qquad then: 0\\
& \quad else \mbox{   } if (N=0)  \qquad \qquad \qquad \qquad \qquad \qquad \qquad \quad then: 1 \\
& \quad else \mbox{   } if((k=0) \mbox{ }or\mbox{ } (l=0)) \qquad \qquad \qquad \qquad \qquad then: 0 \\
& \quad else \mbox{   } if(k>N) \mbox{   } \qquad \qquad \qquad \qquad \qquad \qquad then: |P(N,N,l,s)| \\
& \quad else \mbox{   } if(l>N) \mbox{   } \qquad \qquad \qquad \qquad \qquad \qquad then: |P(N,k,N,s)| \\
& \quad else \mbox{   } if(s>N) \mbox{   } \qquad \qquad \qquad \qquad \qquad \qquad then: |P(N,k,l,N)| \\
& \quad else: \mbox{   } |P(N,k-1,l,s)|+|P(N,k,l-1,s)|-|P(N,k-1,l-1,s)| \\
& \quad \qquad +|P(N-k-l+1,k-1,l-1,s+l-k-1)| 
\end{aligned}
\end{equation*}
\end{thm}

\begin{algorithm}
	\DontPrintSemicolon 
	\KwIn{A positive integer $n$}
	\KwOut{$|D(n)|$}
	$N \gets n(n-1)$\;
	Allocate a four dimensional array $P[N-n+1][n-1][n][N]$\;
	Fill in the array $P$ using dynamic programming based on Theorem \ref{thm_partition}\;
	$S \gets 0$\;
	\For{$i \in I_e(n)$ } {
		\For{$j \gets 1$ \textbf{to} $\min ({n-1,i-n+1})$ } {
				$S \gets S+P[i-j-n+1][j-1][n-1][n-j-1]$\;
		}
	}
	\Return{$S$}\;
	\caption{Basic Algorithm to compute $|D(n)|$. The index set $I_e(n)=\{N|n\le N\le n(n-1), N \mbox{ is an even integer}\}$.}
	\label{algo:basic}
\end{algorithm}

With the above definitions and techniques, we present our basic algorithm to calculate $|D(n)|$ as
pseudo-codes in Algorithm \ref{algo:basic}.
The declaration of the array $P$ in line 2 specifies the sizes of the four dimensions. When filling in and accessing
the array elements we use the convention that array indices start from 0. Also note that based on the recurrence
in Theorem \ref{thm_partition} the $|P(*,*,l,*)|$ values only depend on the $|P(*,*,l-1,*)|$ values, and for
the purpose of calculating $|D(n)|$ only the $|P(*,*,n-1,*)|$ values are used. These
make it feasible to allocate only size 2 for the third dimension of the array $P$ to save memory. For each fixed $i$ in line 5,
the sum accumulated in the inner loop from lines 6 to 7 is the value $|G^{'}(i,n)|$ based on Equation (\ref{eqn:graphical_partition_detailed})
and it can also be saved as additional output if desired. The integer sequence $|D(n)|$ is in the OEIS
entry \textcolor{red}{\href{https://oeis.org/A095268}{A095268}}.

We note that Equation (\ref{eqn:deg_seq_combine}) and Algorithm \ref{algo:basic} are obtained
by partitioning the set $D(n)$ into disjoint subsets $G'(N,n), N\in I_e(n)$ based on degree sums. Alternatively, we
can partition the set $D(n)$ into disjoint subsets based on the largest term in each degree sequence.
This way effectively only produces a formula that interchanges the two sum symbols in the double sums
in Equation (\ref{eqn:deg_seq_combine1}) and Equation (\ref{eqn:deg_seq_combine}). The change that can be made to Algorithm
\ref{algo:basic} based on the equivalent new formula is to simply interchange the two for loop lines 5 and 6, making the sums accumulated in the
inner loop the number of degree sequences of given length $n$ and fixed largest term $j$. Thus, we can see that
the set of degree sequences of fixed length and fixed sum and the set of degree sequences of fixed length
and fixed largest term can all be efficiently counted.
\subsection{Some Novel Recurrences and the Improved Algorithm}
\label{subsec:recur}
In this section we first review some old and also derive novel recurrences involving the sizes of $D(n)$
and various related sets.
We then present an improved algorithm to compute $|D(n)|$ based on these recurrences.

Define $H(n)$ to be the set of degree sequences of length $n$ with largest part exactly $n-1$ and $L(n)$
the set of degree sequences of length $n$ with largest part less than $n-1$. Clearly $D(n)$ is the disjoint
union of $H(n)$ and $L(n)$. Let $D_0(n)$ denote
the set of degree sequences of length $n$ allowing zero terms. The OEIS entry \textcolor{red}{\href{https://oeis.org/A004251}{A004251}}
records the values of $|D_0(n)|$. 
We now present two propositions demonstrating the relationships among the cardinalities of these sets that will
be used by our improved algorithm. The following proposition is not hard to
derive by the simple observation that $D_0(n)$ is the disjoint union of
$D(n)$ and the set of degree sequences of length $n$ containing at least one zero.

\begin{proposition}
\label{thm_deg_seq} \cite[Lemma 2]{Ivanyi2013}: $|D_0(n)|=|D_0(n-1)|+|D(n)|$, for $n\ge 2$. And
\[ |D_0(n)|=\left\{ \begin{array}{ll}
1 & \mbox{if $n=1$};\\
1+\sum_{i=2}^{n}|D(i)| & \mbox{if $n \ge 2$}.\end{array} \right.\]
\end{proposition}

Based on Proposition \ref{thm_deg_seq} it is a trivial thing to compute $|D_0(n)|$ once all $|D(i)|$ such that $i\le n$ have been computed.
We also have the following proposition that transforms the problem of computing $|D(n)|$ to that
of computing $|L(n)|$.

\begin{proposition}
\label{thm_deg_seq_high_low} $|H(n)|=|D_0(n-1)|$, and $|D(n)|=|D_0(n-1)|+|L(n)|$, for $n\ge 2$.
\end{proposition}
\begin{proof}
Each simple undirected graph $G$ with $n$ vertices having a vertex of degree $n-1$ can be transformed into
a graph $G_1$ of $n-1$ vertices (possibly with isolated vertices) by removing the vertex of degree $n-1$.
Thus, $H(n)\subseteq D_0(n)$.

Conversely, each graph $G_1$ of $n-1$ vertices (possibly with isolated vertices) can be transformed into a graph $G$ with $n$
vertices having a vertex of degree $n-1$ by adding a new vertex and making it adjacent to all the vertices originally in $G_1$.
Thus, $D_0(n)\subseteq H(n)$. This bijection establishes that $|H(n)|=|D_0(n-1)|$.

Since $D(n)=H(n)\cup L(n)$, and this is a disjoint union, we have $|D(n)|=|H(n)|+|L(n)|=|D_0(n-1)|+|L(n)|.$
\end{proof}
	
By Propositions \ref{thm_deg_seq} and \ref{thm_deg_seq_high_low}, it is sufficient to compute $|L(n)|$ in order to
compute $|D(n)|$ when all $|D(i)|$ such that $i<n$ are already known. Not surprisingly, we will show that $|L(n)|$
can also be computed eventually through the $|P(*,*,*,*)|$ values.

Define $L^{'}(N,l)$ to be the set of graphical partitions of $N$
with exactly $l$ parts and largest part less than $l-1$. Also define an index set $J_e(n)=\{N|n\le N\le n(n-2), N \mbox{ is an even integer}\}$. It is easily seen that	
\begin{equation} \label{eqn:lower_sum_set}
	\begin{aligned}
	L(n)=\bigcup_{N\in J_e(n)}L^{'}(N,n)
	\end{aligned}
\end{equation}	
and this is a disjoint union so that we have
\begin{equation} \label{eqn:lower_sum}
\begin{aligned}
|L(n)|=\sum_{N\in J_e(n)}|L^{'}(N,n)|.
\end{aligned}
\end{equation}

By definition we have
\begin{equation} \label{eqn:lower_sum_partition_set}
\begin{aligned} 
L^{'}(N,n)=\bigcup_{k=1}^{n-2}G^{'}(N,k,n)
\end{aligned}
\end{equation}
and by disjointness of this union we have
\begin{equation} \label{eqn:lower_sum_partition}
\begin{aligned}
|L^{'}(N,n)|=\sum_{k=1}^{n-2}|G^{'}(N,k,n)|.
\end{aligned}
\end{equation}
Combining Equations (\ref{eqn:lower_sum}) and (\ref{eqn:lower_sum_partition}) and substituting the formula in Theorem \ref{thm_graphical_partition} into Equation (\ref{eqn:lower_sum_partition}) we get
\begin{equation} \label{eqn:lower_sum_final}
\begin{aligned} 
\begin{split}
|L(n)|&=\sum_{N\in J_e(n)}\sum_{k=1}^{n-2}|G^{'}(N,k,n)|\\
&=\sum_{N\in J_e(n)}\sum_{k=1}^{n-2}|P(N-k-n+1,k-1,n-1,n-k-1)|
\end{split}
\end{aligned}
\end{equation}
and this shows how $|L(n)|$ can also be calculated by summing a finite number of $|P(*,*,*,*)|$ values.
In fact, the calculation of $|L(n)|$ can be further sped up by the following proposition.

\begin{proposition}
	\label{thm_lower_complementary} $ |L^{'}(N,n)|=|L^{'}(n(n-1)-N,n)|$, for $N \in J_e(n)$.
\end{proposition}
\begin{proof}
Each degree sequence $d_1\ge d_2 \ge \ldots \ge d_n >0$ in $L(n)$ with sum $N$ and $d_1<n-1$ can be transformed into a
degree sequence $n-1-d_n\ge n-1-d_{n-1} \ge \ldots \ge n-1-d_1 >0$ with sum $n(n-1)-N$, which still lies
in $L(n)$ since $n-1-d_n<n-1$.
\end{proof}

Based on Equation (\ref{eqn:lower_sum}) and Proposition \ref{thm_lower_complementary}, it is sufficient to
calculate the $|L^{'}(N,n)|$ values for even $N$ between $n$ and $n(n-1)/2$ in order to calculate $|L(n)|$,
apparently saving the work load by a constant factor.

\begin{algorithm}
	\DontPrintSemicolon 
	\KwIn{A positive integer $n$}
	\KwOut{$|D(n)|$}
	$N \gets n(n-1)/2$\;
	Allocate a four dimensional array $P[N-n+1][n-2][n][N]$\;
	Fill in the array $P$ using dynamic programming based on Theorem \ref{thm_partition}\;
	$S \gets 0$\;
	\For{$i \in J'_e(n)$ } {
		\For{$j \gets 1$ \textbf{to} $\min ({n-2,i-n+1})$ } {
			$S \gets S+P[i-j-n+1][j-1][n-1][n-j-1]$\;
		}
	}
	$S \gets 2S$\;
	\If{$N$ is even}	{ 
		\For{$j \gets 1$ \textbf{to} $\min ({n-2,N-n+1})$ } {
			$S \gets S+P[N-j-n+1][j-1][n-1][n-j-1]$\;
		}
	}
	Calculate $|D_0(n-1)|$ using saved values of $|D(i)|,i<n$ by Proposition \ref{thm_deg_seq}\;
	\Return{$S+|D_0(n-1)|$}\;
	\caption{Improved Algorithm to compute $|D(n)|$, assuming all $|D(i)|$ such that $i<n$ are already computed. The
		index set $J'_e(n)=\{N|n\le N< n(n-1)/2, N \mbox{ is an even integer}\}$.}
	\label{algo:improved}
\end{algorithm}

We now present an improved Algorithm \ref{algo:improved} to calculate $|D(n)|$ based on the above definitions
and techniques, assuming all values of $|D(i)|$ such that $i<n$ have been calculated.
The index set $J'_e(n)$
used in the improved algorithm is defined as $J'_e(n)=\{N|n\le N< n(n-1)/2, N \mbox{ is an even integer}\}$.
The variable $S$ is used to store the value of $|L(n)|$.
As in Algorithm \ref{algo:basic}, the size for the third dimension of the array $P$ can also be reduced to 2 for
the purpose of computing $|L(n)|$ and $|D(n)|$. The individual sums accumulated in the loop
from lines 6 to 7 and lines 10 to 11 are the values $|L^{'}(*,n)|$ based on Equation (\ref{eqn:lower_sum_partition})
and Equation (\ref{eqn:lower_sum_final})
and they together with the $|L^{'}(n(n-1)-*,n)|$ values can also be saved as additional output if desired.
Algorithm \ref{algo:improved}
is called an improved algorithm because it achieves a constant factor improvement over Algorithm \ref{algo:basic}
(see complexity analysis in Section \ref{sec:Analysis}). Although using more time and space, Algorithm \ref{algo:basic}
does have the advantage of applicability for any $n$ without the need to know all the values of $|D(i)|$ such that $i<n$
as in Algorithm \ref{algo:improved}.

Before concluding this section we derive one more interesting recurrence. Define
$H^{'}(N,l)$ to be the set of graphical partitions of $N$
with exactly $l$ parts and largest part exactly $l-1$. Also define an index set
$I'_e(n)=\{N|2(n-1)\le N\le n(n-1), N \mbox{ is an even integer}\}$. By definitions it is easily seen that
\begin{equation} \label{eqn:higher_sum_set}
\begin{aligned}
H(n)=\bigcup_{N\in I'_e(n)}H^{'}(N,n)
\end{aligned}
\end{equation}
and this is a disjoint union so that we have
\begin{equation} \label{eqn:higher_sum}
\begin{aligned}
|H(n)|=\sum_{N\in I'_e(n)}|H^{'}(N,n)|.
\end{aligned}
\end{equation}

By definition we have
\begin{equation} \label{eqn:higher_sum_partition_set}
\begin{aligned}
H^{'}(N,n)=G^{'}(N,n-1,n),
\end{aligned}
\end{equation}
and
\begin{equation} \label{eqn:higher_lower_partition_set}
\begin{aligned}
G^{'}(N,n)=H^{'}(N,n) \cup L^{'}(N,n), \qquad |G^{'}(N,n)|=|H^{'}(N,n)|+ |L^{'}(N,n)|.
\end{aligned}
\end{equation}

The following proposition indicates a symmetry of $|H^{'}(N,n)|$ similar to that of $|L^{'}(N,n)|$ as shown
in Proposition \ref{thm_lower_complementary}.
\begin{proposition}
	\label{thm_higher_complementary}
	$|H^{'}(N,n)|=|H^{'}((n+2)(n-1)-N,n)|$, for $N \in I'_e(n)$.
\end{proposition}
\begin{proof}
Each degree sequence $d_1\ge d_2 \ge \ldots \ge d_n >0$ in $H(n)$ with sum $N$ and $d_1=n-1$ can be first transformed into a
degree sequence allowing zero terms (of length $n$, by appending a zero after $d_n-1$)
$d_2-1 \ge d_3-1 \ge \ldots \ge d_n-1 \ge 0 \ge 0$ by applying
the Havel-Hakimi algorithm \cite{Havel1955,Hakimi1962} which
characterizes equivalent realizable sequences, and then transformed into a degree sequence (of length $n$)
$n-1 \ge n-d_n\ge n-d_{n-1} \ge \ldots \ge n-d_2 >0$ with sum $(n+2)(n-1)-N$, which still lies
in $H(n)$.
\end{proof}

From available numerical results we observe that for fixed $n$ the sequence of $|H^{'}(N,n)|$ for $N\in I'_e(n)$
is unimodal and the maximum value in the sequence occurs when $N$ is an even integer closest to $(n^2+n-2)/2$.
It is also observed that for fixed $n$ the sequence of $|L^{'}(N,n)|$ for $N\in J_e(n)$ is unimodal and
the maximum value in the sequence occurs when $N$ is an even integer closest to $(n^2-n)/2$. Based on Equation
(\ref{eqn:higher_lower_partition_set}) and available numerical results we also conjecture that for fixed $n$
the sequence of $|G^{'}(N,n)|$ for $N\in I_e(n)$
is unimodal and the maximum value in the sequence occurs when $N$ is closest to $n^2/2-n/6$. One of the ways to determine the
asymptotic order of $|D(n)|$ might be to determine the order of $|G^{'}(N,n)|$ for a narrow range of $N$ centered around $n^2/2-n/6$.
The rigorous proof of these results might need advanced techniques.

\section{Algorithms to compute the number of degree sequences of connected graphs on $n$ vertices}
\label{sec:count_connected}
In this section we apply the techniques employed in the two algorithms for computing $|D(n)|$ in Section \ref{sec:Algo} to derive similar
algorithms for computing the number of degree sequences of connected graphs on $n$ vertices.

Recall that a degree sequence is called \textit{potentially} connected (resp. disconnected) if it has a connected
(resp. disconnected) realization and it is called \textit{forcibly} connected (resp. disconnected) if all of
its realizations are connected (resp. disconnected).
Let $D_c(n)$ denote the subset of $D(n)$ that are potentially connected and $D_d(n)$ the subset of $D(n)$ that are
forcibly disconnected. Clearly $D(n)$ is the disjoint union of $D_c(n)$ and $D_d(n)$. A degree sequence
$d=(d_1,d_2,\cdots,d_n)$ in $D(n)$ is potentially connected if and only if $\sum_{i=1}^{n}d_i\ge 2(n-1)$ \cite{Berge1973}.
Because of this simple characterization it is easily seen by definitions that
\begin{equation} \label{eqn:degseq_set_potentially_connected}
\begin{aligned}
D_c(n)=\bigcup_{N\in I'_e(n)}G^{'}(N,n)=\bigcup_{N\in I'_e(n)}\bigcup_{k=1}^{n-1}G^{'}(N,k,n),
\end{aligned}
\end{equation}

and

\begin{equation} \label{eqn:degseq_set_forcibly_disconnected}
\begin{aligned}
D_d(n)=\bigcup_{N\in I_e(n)-I'_e(n)}G^{'}(N,n)=\bigcup_{N\in I_e(n)-I'_e(n)}\bigcup_{k=1}^{n-1}G^{'}(N,k,n).
\end{aligned}
\end{equation}
Again by disjointness we have
\begin{equation} \label{eqn:degseq_potentially_connected}
\begin{aligned}
\begin{split}
|D_c(n)|&=\sum_{N\in I'_e(n)}\sum_{k=1}^{n-1}|G^{'}(N,k,n)|\\
&=\sum_{N\in I'_e(n)}\sum_{k=1}^{n-1}|P(N-k-n+1,k-1,n-1,n-k-1)|
\end{split}
\end{aligned}
\end{equation}

and

\begin{equation} \label{eqn:degseq_forcibly_disconnected}
\begin{aligned}
\begin{split}
|D_d(n)|&=\sum_{N\in I_e(n)-I'_e(n)}\sum_{k=1}^{n-1}|G^{'}(N,k,n)|\\
&=\sum_{N\in I_e(n)-I'_e(n)}\sum_{k=1}^{n-1}|P(N-k-n+1,k-1,n-1,n-k-1)|.
\end{split}
\end{aligned}
\end{equation}
If $|D(n)|$ has already been obtained, we can choose to compute $|D_d(n)|$ in order to get $|D_c(n)|=|D(n)|-|D_d(n)|$
for performance reasons. Now we present Algorithm \ref{algo:potentially_connected} to compute $|D_c(n)|$ directly
based on Equation (\ref{eqn:degseq_potentially_connected}) and Algorithm \ref{algo:forcibly_disconnected}
to compute $|D_c(n)|$ indirectly through computing $|D_d(n)|$ based on Equation (\ref{eqn:degseq_forcibly_disconnected}).
As in Algorithms \ref{algo:basic} and \ref{algo:improved}
we can still allocate size 2 for the third dimension of the array $P$.

\begin{algorithm}[h]
	\DontPrintSemicolon 
	\KwIn{A positive integer $n$}
	\KwOut{$|D_c(n)|$}
	$N \gets n(n-1)$\;
	Allocate a four dimensional array $P[N-n+1][n-1][n][N]$\;
	Fill in the array $P$ using dynamic programming based on Theorem \ref{thm_partition}\;
	$S \gets 0$\;
	\For{$i \in I'_e(n)$ } {
		\For{$j \gets 1$ \textbf{to} $\min ({n-1,i-n+1})$ } {
			$S \gets S+P[i-j-n+1][j-1][n-1][n-j-1]$\;
		}
	}
	\Return{$S$}\;
	\caption{Algorithm to compute $|D_c(n)|$ directly. The index set $I'_e(n)=\{N|2(n-1)\le N\le n(n-1), N \mbox{ is an even integer}\}$.}
	\label{algo:potentially_connected}
\end{algorithm}

\begin{algorithm}[h]
	\DontPrintSemicolon 
	\KwIn{A positive integer $n$}
	\KwOut{$|D_c(n)|$}
	$N \gets 2(n-2)$\;
	Allocate a four dimensional array $P[N-n+1][n-1][n][N]$\;
	Fill in the array $P$ using dynamic programming based on Theorem \ref{thm_partition}\;
	$S \gets 0$\;
	\For{$i \in I_e(n)-I'_e(n)$ } {
		\For{$j \gets 1$ \textbf{to} $\min ({n-1,i-n+1})$ } {
			$S \gets S+P[i-j-n+1][j-1][n-1][n-j-1]$\;
		}
	}
	\Return{$|D(n)|-S$}\;
	\caption{Algorithm to compute $|D_c(n)|$ indirectly, assuming $|D(n)|$ is known. The index set
		$I_e(n)-I'_e(n)=\{N|n\le N< 2(n-1), N \mbox{ is an even integer}\}$.}
	\label{algo:forcibly_disconnected}
\end{algorithm}

The OEIS entry \textcolor{red}{\href{https://oeis.org/A007721}{A007721}} records the values of $|D_c(n)|$. 

\section{Algorithm to compute the number of degree sequences of biconnected graphs on $n$ vertices}
\label{sec:count_biconnected}
In this section we will derive Algorithm \ref{algo:potentially_biconnected}
to count the set $D_b(n)$ of degree sequences of biconnected graphs on $n$ vertices.
A degree sequence $d_1\ge d_2 \ge \cdots \ge d_n$ in $D(n)$ is potentially biconnected if and only if
$\sum_{i=1}^{n}d_i\ge 2n-4+2d_1$ and $d_n\ge 2$ \cite{WangKleitman1973}. Since this characterization
expresses the lower bound of the degree sum as a quantity that can vary due to $d_1$, there is not a simple
algorithm like Algorithm \ref{algo:potentially_connected} or Algorithm \ref{algo:forcibly_disconnected} to count $|D_b(n)|$.
However, we will show that $|D_b(n)|$ can still be computed eventually through the $|P(*,*,*,*)|$ values
and the unrestricted partition function $|P(*)|$ values.

Define $S(n)$ to be the set of degree sequences of length $n$ with largest part exactly $n-2$, $C(n)$ the set of degree
sequences of length $n$ with smallest part exactly $1$ and $B(n)$ the set
of degree sequences of length $n$ with largest part exactly $n-1$ and smallest part exactly $1$. Then, $D_2(n)=D(n)-C(n)$
is the set of degree sequences of length $n$ with smallest part at least 2. We next establish some connections
between these sets and show that $|D_2(n)|$ can be efficiently computed.

\begin{proposition}
	\label{thm:deg_seq_second_and1} $|C(n)|=|B(n)|+|S(n)|$,$|B(n)|=|D_0(n-2)|$, for $n\ge 3$.
\end{proposition}
\begin{proof}
	Notice that $C(n)$ is the disjoint union of $B(n)$ and the set $W(n)$ of degree sequences of length $n$ with largest part
	less than $n-1$ and smallest part exactly 1, so we have $|C(n)|=|B(n)|+|W(n)|$. Each degree sequence 
	$d_1\ge d_2 \ge \ldots \ge d_n$ in $W(n)$
	can be transformed into a degree sequence $n-1-d_n \ge n-1-d_{n-1} \ge \ldots \ge n-1-d_1$, which is a
	degree sequence in $S(n)$ since $d_1<n-1$ and $d_n=1$. Conversely, each degree sequence
in $S(n)$ can be transformed into one in $W(n)$ in a similar way. This bijection
	establishes that $|W(n)|=|S(n)|$, and then the first equation follows.
	
	Each simple undirected graph $G$ with $n$ vertices having a vertex of degree $n-1$ and a vertex of degree 1
	can be transformed into
	a graph $G_2$ of $n-2$ vertices (possibly with isolated vertices) by first removing the vertex of degree $n-1$, resulting
	in an isolated vertex (previously with degree 1), and then removing the newly appeared isolated vertex. Conversely,
	each graph $G_2$ of $n-2$ vertices (possibly with isolated vertices) can be transformed to a graph $G$ with $n$
	vertices having a vertex of degree $n-1$ and a vertex of degree 1 by first adding a new isolated vertex and then
	adding a second vertex and making it adjacent to all the other $n-1$ vertices. This bijection establishes that
	$|B(n)|=|D_0(n-2)|$.
	
\end{proof}

Now $|B(n)|$ can be efficiently computed based on Proposition \ref{thm:deg_seq_second_and1} and Proposition \ref{thm_deg_seq}.
As noted at the end of Section \ref{subsec:BasicAlgo}, $|S(n)|$ can be efficiently computed. Each
degree sequence in $S(n)$ must have sum $N$ such that $2n-3\le N\le n(n-2)$. Define an index set
$I''_e(n)=\{N|2n-3\le N\le n(n-2), N \mbox{ is an even integer}\}$. Clearly,
\[S(n)=\bigcup_{N\in I''_e(n)}G^{'}(N,n-2,n).\]	
By disjointness, we have
\begin{equation} \label{eqn:second_largest}
|S(n)|=\sum_{N\in I''_e(n)}|G^{'}(N,n-2,n)|=\sum_{N\in I''_e(n)}|P(N-2n+3,n-3,n-1,1)|.
\end{equation}
Thus, $|C(n)|$ can be efficiently computed based on Proposition \ref{thm:deg_seq_second_and1} and we can therefore
efficiently compute $|D_2(n)|$ as
\begin{equation} \label{eqn:D2nfirst}
|D_2(n)|=|D(n)|-|C(n)|.
\end{equation}
Next we will show that the combinatorial trick
of expressing $|D_2(n)|$ in a different way allows us to compute $|D_b(n)|$ efficiently.

Recall that $D_b(n)$ is the set of degree sequences of biconnected graphs on $n$ vertices. Clearly $D_b(n)\subseteq D_2(n)$
based on the characterization of potentially biconnected degree sequences \cite{WangKleitman1973}.
Denote $D_{2\setminus b}(n)=D_2(n)-D_b(n)$, i.e., the set of those degree sequences $d_1\ge d_2\ge \cdots \ge d_n$ such that
$\sum_{i=1}^{n}d_i< 2n-4+2d_1$ and $d_n\ge 2$. Clearly, we can express $|D_2(n)|$ in a second way as
\begin{equation} \label{eqn:D2nsecond}
|D_2(n)|=|D_b(n)|+|D_{2\setminus b}(n)|.
\end{equation}
This fulfills our goal of efficiently computing $|D_b(n)|=|D_2(n)|-|D_{2\setminus b}(n)|$
as long as $|D_{2\setminus b}(n)|$ can be computed efficiently, which will be shown next.

Consider the possible degree sequences $d_1\ge d_2\ge \cdots \ge d_n$ in $D_{2\setminus b}(n)$.
Since $D_{2\setminus b}(n) \subseteq D_2(n)$, $d_1\ge d_n \ge 2$. The possible choices for $d_1$ are $d_1=2,\cdots,n-1$
and they are considered individually below.

\begin{enumerate}
	\item $d_1=2:$ It is required that $\sum_{i=1}^{n}d_i< 2n$. This is impossible since we also need $\sum_{i=1}^{n}d_i\ge 2(n-1)+2=2n$.
	\item $d_1=3:$ It is required that $\sum_{i=1}^{n}d_i< 2n+2$. Again impossible since we also need $\sum_{i=1}^{n}d_i\ge 2(n-1)+3=2n+1$.

\item $d_1=2k\ge 4:$ It is required that $\sum_{i=1}^{n}d_i< 2n+4k-4$. We also need $\sum_{i=1}^{n}d_i\ge 2(n-1)+2k=2n+2k-2$. Thus
$\sum_{i=1}^{n}d_i=2n+2k-2,2n+2k,\ldots,2n+4k-6$ are all the feasible degree sums. For each feasible degree sum $2n+2k-2+N$,
$N=0,2,\cdots,2k-4$, consider the possible degree sequences $d_1\ge d_2\ge \cdots \ge d_n$ such that $d_1=2k$, $d_n\ge 2$, $\sum_{i=1}^{n}d_i=2n+2k-2+N$.
Clearly the number of possible degree sequences is the number of ways to partition $N$ into $t\le n-1$ parts
$N=N_1+N_2+\cdots+N_t$ such that
$d_1\ge 2+N_1\ge 2+N_2\ge \cdots \ge 2+N_t \ge 2\ge \cdots \ge 2$ is a graphical degree sequence.
The particular case $t=1$ gives the integer sequence $d_1\ge 2+N\ge 2\ge \cdots \ge 2$.
It is easy to check based on the Nash-Williams condition \cite{Ruch1979,Rousseau1995,Sierksma1991}
that for each $N=0,2,\cdots,2k-4$, the integer sequence $d_1\ge 2+N\ge 2\ge \cdots \ge 2$ is graphical.
For each fixed $N\in \{0,2,\cdots,2k-4\}$ and any partition $N=N_1+N_2+\cdots+N_t$, the integer sequence
$l_1=(d_1, 2+N, 2, \cdots ,2)$ majorizes the integer sequence $l_2=(d_1, 2+N_1, 2+N_2, \cdots ,2+N_t, 2, \cdots ,2)$ and
since $l_1$ is graphical, $l_2$ is also graphical \cite[Corollary 3.1.4]{MahadevPeled1995}.
This produces a total of $|P(0)|+|P(2)|+|P(4)|+\cdots+|P(2k-4)|$ possible degree sequences.

	\item $d_1=2k+1\ge 5:$ It is required that $\sum_{i=1}^{n}d_i< 2n+4k-2$. We also need $\sum_{i=1}^{n}d_i\ge 2(n-1)+2k+1=2n+2k-1$. Thus
	$\sum_{i=1}^{n}d_i=2n+2k,2n+2k+2,\ldots,2n+4k-4$ are all the feasible degree sums.
	For each feasible degree sum $2n+2k-1+N$,
	$N=1,3,\cdots,2k-3$, consider the possible degree sequences $d_1\ge d_2\ge \cdots \ge d_n$ such that $d_1=2k+1$, $d_n\ge 2$, $\sum_{i=1}^{n}d_i=2n+2k-1+N$.
	Clearly the number of possible degree sequences is the number of ways to partition $N$ into $t\le n-1$ parts
	$N=N_1+N_2+\cdots+N_t$ such that
	$d_1\ge 2+N_1\ge 2+N_2\ge \cdots \ge 2+N_t \ge 2\ge \cdots \ge 2$ is a graphical degree sequence.
	The particular case $t=1$ gives the integer sequence $d_1\ge 2+N\ge 2\ge \cdots \ge 2$.
	It is easy to check based on the Nash-Williams condition \cite{Ruch1979,Rousseau1995,Sierksma1991}
	that for each $N=1,3,\cdots,2k-3$, the integer sequence $d_1\ge 2+N\ge 2\ge \cdots \ge 2$ is graphical.
	For each fixed $N\in \{1,3,\cdots,2k-3\}$ and any partition $N=N_1+N_2+\cdots+N_t$, the integer sequence
	$l_1=(d_1, 2+N, 2, \cdots ,2)$ majorizes the integer sequence $l_2=(d_1, 2+N_1, 2+N_2, \cdots ,2+N_t, 2, \cdots ,2)$ and
	since $l_1$ is graphical, $l_2$ is also graphical \cite[Corollary 3.1.4]{MahadevPeled1995}.
	This produces a total of $|P(1)|+|P(3)|+|P(5)|+\cdots+|P(2k-3)|$ possible degree sequences.
	
\end{enumerate}

Based on the above discussions and techniques, we present Algorithm \ref{algo:potentially_biconnected} to
compute the number of potentially biconnected degree sequences of length $n$. The variable $S$ is supposed
to store the value $|D_{2\setminus b}(n)|$. The OEIS entry \textcolor{red}{\href{https://oeis.org/A007722}{A007722}} records the $|D_b(n)|$ values.
\begin{algorithm}
	\DontPrintSemicolon 
	\KwIn{A positive integer $n\ge 5$}
	\KwOut{$|D_b(n)|$}
	Compute $|S(n)|$ using dynamic programming based on Equation (\ref{eqn:second_largest})\;
	Compute $|D_0(n-2)|$ based on Proposition \ref{thm_deg_seq} using the saved $|D(i)|$ values\;
	Compute $|C(n)|=|D_0(n-2)|+|S(n)|$ based on Proposition \ref{thm:deg_seq_second_and1}\;
	Compute $|D_2(n)|=|D(n)|-|C(n)|$ using the saved $|D(n)|$ value\;
	$S \gets 0$\;
	\For{$i \gets 4$ \textbf{to} $n-1$} {
		$k \gets \lfloor \frac{i}{2} \rfloor$\;
		\uIf{$i$ is even}{
			\For{$j \gets 0$ \textbf{to} $2k-4$ \textbf{step} $2$ } {
				$S \gets S+|P(j)|$\;
			}
		}
		\Else{
			\For{$j \gets 1$ \textbf{to} $2k-3$ \textbf{step} $2$ } {
				$S \gets S+|P(j)|$\;
			}
		}
	}
	\Return{$|D_2(n)|-S$}\;
	\caption{Algorithm to compute $|D_b(n)|$, assuming all $|D(i)|$ such that $i\le n$ have been computed.}
	\label{algo:potentially_biconnected}
\end{algorithm}


Before we conclude this section, we briefly mention that the method in this section can be extended to count
the number of degree sequences of triconnected graphs and beyond. Characterization of potentially $k$-connected
degree sequences of length $n$ can be found in Wang and Cleitman \cite{WangKleitman1973}. In order to use
this characterization, we can still count the set $D_k(n)$ of degree sequences of length $n$ with smallest
part at least $k$ in two different ways. Let $C_k(n)$ denote the set of degree sequences of length $n$ with smallest
part exactly $k$. Clearly $|D_k(n)|=|D(n)|-\sum_{i=1}^{k-1}|C_i(n)|$.
The first way expresses $|C_i(n)|$ as the sum of the number of degree
sequences with fixed largest part $n-i-1$ and lengths from $n-i$ to $n$ so that $|D_k(n)|$ can be efficiently computed.
The second way expresses $|D_k(n)|$
as the sum of the number $|D_{kc}(n)|$ of potentially $k$-connected degree sequences and the rest $|R(n)|$. Then
$|D_{kc}(n)|$ can be efficiently computed by combining the results of computing $|D_k(n)|$ in the two ways. Counting $R(n)$ in $D_k(n)$
involves an enumeration of all possible values of $d_1,\cdots, d_{k-1}$ and will become the most tedious
part of the algorithm.

\section{Complexity Analysis}
\label{sec:Analysis}
As mentioned in Section \ref{sec:Intro}, the previous best known methods to calculate $|D(n)|$ have a time
complexity of $\Omega(4^n/n)$ as in \cite{Ruskey1994} and a time complexity of $O(\sqrt{n} 4^n)$ and space
complexity of $O(n)$ as in \cite{Ivanyi2013} respectively. Their exponential running time comes from the fact
that they compute $|D(n)|$ by generating all the sequences in $D(n)$. The variants of these algorithms to compute
$|D_c(n)|$ and $|D_b(n)|$ by incorporating the relevant tests whether the degree sequence is potentially connected
or potentially biconnected have similar exponential time
complexity since the tests only add linear extra time for each generated degree sequence. Thus it is infeasible to
calculate $|D(n)|$, $|D_c(n)|$ and $|D_b(n)|$ using these methods for $n$ beyond 35. In this section we analyze the computational
complexity of the presented algorithms. A summary of the results is presented in Table \ref{tbl:complexity}.

\begin{table}[!htb]
	\centering
	\caption{Complexity of the algorithms}
	\begin{tabular}{||c|c|c||}
		\hline\hline
		Algorithm & Time & Space\\
		\hline\hline
		Ruskey \textit{et al.}\cite{Ruskey1994} & $\Omega(4^n/n)$ & $O(n)$\\
		\hline
		Iv\'{a}nyi \textit{et al.}\cite{Ivanyi2013} & $\Theta(\sqrt{n} 4^n)$ & $O(n)$\\
		\hline
		Algorithm \ref{algo:basic} & $O(n^5)$ & $O(n^5)$ \\
		\hline
		Algorithm \ref{algo:improved} & $O(n^5)$ & $O(n^5)$ \\
		\hline
		Algorithm \ref{algo:potentially_connected} & $O(n^5)$ & $O(n^5)$ \\
		\hline
		Algorithm \ref{algo:forcibly_disconnected} & $O(n^3)$ & $O(n^3)$ \\
		\hline
		Algorithm \ref{algo:potentially_biconnected} & $O(n^5)$ & $O(n^5)$ \\
		\hline\hline
	\end{tabular}
	\label{tbl:complexity}
\end{table}

We first analyze the complexity of the basic Algorithm \ref{algo:basic}. The first and fourth dimensions of
the array $P$ have sizes $O(n^2)$, the second dimension $O(n)$ and the third dimension a constant as mentioned above.
Thus, the space complexity is $O(n^5)$. The dynamic programming step in line 3 costs time $O(n^5)$ as a result.
The loop from lines 5 to 7 costs time $O(n^2 n)=O(n^3)$, which is dominated by the step in line 3. So Algorithm
\ref{algo:basic} has time complexity $O(n^5)$.

Similar analysis shows that the improved Algorithm \ref{algo:improved} for $|D(n)|$
and Algorithm \ref{algo:potentially_connected} for $|D_c(n)|$ also have the asymptotic time and space
complexity $O(n^5)$. For Algorithm \ref{algo:improved}, since the first and fourth dimensions of the array $P$
have sizes about half of those
in the basic Algorithm \ref{algo:basic}, it achieves an improvement of a constant factor of approximately 4.
Although it is not asymptotically faster than the basic algorithm, it does allow us to compute more $|D(n)|$
values since for large values of $n$ the memory issue is more severe than the time issue for these algorithms.
Algorithm \ref{algo:forcibly_disconnected} has time and space complexity $O(n^3)$ due to the reduction
of the sizes of the first and fourth dimension of the array $P$ to $O(n)$.

For Algorithm \ref{algo:potentially_biconnected}, line 1 takes $O(n^5)$ time and $O(n^5)$ space. Lines 2 to 4
take time $O(n)$ using saved values of $|D(i)|,i\le n$. For lines 5 to 13, all $|P(j)|,j\le n-5$ can be precomputed
in time $O(n^{3/2+\epsilon})$ since $|P(n)|$ can be computed in time $O(n^{1/2+\epsilon})$ \cite{Johansson2012}. The double for loop
from lines 6 to 13 takes time $O(n^2)$. The overall time and space complexity are both $O(n^5)$.

\section{Computational Experiences}
\label{sec:experiments}

We note that our calculation of $|D(30)|$ using an optimized C implementation of the basic Algorithm \ref{algo:basic}
takes a single Intel Core i7 computer with 16 GB memory 1.8 seconds. The improved Algorithm \ref{algo:improved} uses 0.7 seconds
to get the same result. In contrast, the previous result of \cite{Ivanyi2013} estimated that computing $|D(30)|$
using their method would cost a single computer about 72 years. The computation of $|D(118)|$ using the improved
Algorithm \ref{algo:improved} on a computer with 256 GB memory takes about 2.5 hours.

For $n<30$ our results of $|D(n)|$ agree with those in \cite{Ivanyi2013} while the values of $|D(30)|$ and $|D(31)|$ we calculated are
slightly different. From the available numerical results it appears that $|D(n)|/|D(n-1)|$
is slowly increasing when $n\ge 9$ and approaching 4, giving more supporting evidence of Royle's conjecture.
We also computed the $|D_0(n)|$, $|D_c(n)|$ and $|D_b(n)|$ values and incorporated all available numerical
results into their respective OEIS entries.

During the preparation of this manuscript, the author became aware of the results in \cite{KohnertSlides}, which
calculated the values of $|D(n)|$ up to $n=34$. The method employed therein seemed more complicated and not formally
published to the knowledge of the author.

\section{Conclusions}
\label{sec:conclusion}
In this paper we presented dynamic programming algorithms to compute the number of degree sequences
of simple graphs, simple connected graphs and simple biconnected graphs on $n$ vertices exactly all based on a known recurrence
of Barnes and Savage. They are the first known polynomial time algorithms to compute
these functions and are asymptotically much faster than previous best known counterparts. Using these
fast algorithms we extend the known exact values of $|D(n)|$,  $|D_0(n)|$, $|D_c(n)|$ and $|D_b(n)|$ up to $n=118$ and these
results have been included in their OEIS entries.
Further research include how to determine the tight asymptotic order of these functions, how to efficiently compute
the number of forcibly $k$-connected degree sequences of length $n$ and determine its asymptotic order.

\section{Acknowledgements}
This research has been supported by the Startup Fund of Georgia Southern University.

\sloppy

\bibliographystyle{plain}
\bibliography{efficientDScounting}







\end{document}